\documentclass{amsart}

\usepackage{amsmath,amssymb,amsfonts,amsthm}
\usepackage{color}

\newcommand{\D}{\mathrm {d}}
\def\le{\leqslant}
\def\ge{\geqslant}

\newcommand{\beq}{\begin{equation}}
\newcommand{\eeq}{\end{equation}}

\def\u{u^\eps}
\def\uapp{u_{\rm app}^\eps}
\def\r{r^\eps}
\def\calA{\mathcal A}
\def\calN{\mathcal N}
\def\calM{\mathcal M}

\def\ds{\displaystyle}

\newtheorem{lemma}{Lemma}[section]
\newtheorem{theorem}[lemma]{Theorem}
\newtheorem{proposition}[lemma]{Proposition}

\newtheorem{corollary}[lemma]{Corollary}
\theoremstyle{definition}
\newtheorem{definition}[lemma]{Definition}
\theoremstyle{definition}
\newtheorem{remark}[lemma]{Remark}
\theoremstyle{definition}

\newcommand{\N}{{\mathbb N }}
\newcommand{\R}{{\mathbb R}}
\newcommand{\C}{{\mathbb C}}

\newcommand{\e}{{\varepsilon }}
\newcommand{\eps}{{\varepsilon }}

\def\d{{\partial}}

\def\({\left(}
\def\){\right)}
\def\<{\left\langle}
\def\>{\right\rangle}
\def\O{\mathcal O}

\newcommand{\Id}[1]{{\rm I\kern-2pt I_{#1}}}
\renewcommand{\hbar}{{\displaystyle\bar{\phantom{x}}\kern-6pt h}}

\numberwithin{equation}{section}

\begin{document}


\title[High-frequency averaging in semi-classical Hartree-type equations]
{High-frequency averaging in\\ semi-classical Hartree-type equations} 

\author[J.~Giannoulis]{Johannes Giannoulis}
\author[A.~Mielke]{Alexander Mielke}
\author[C.~Sparber]{Christof Sparber}
\address[J.~Giannoulis]{Zentrum Mathematik, 
Technische Universit\"at M\"unchen\\ Boltz\-mann\-stra{\ss}e 3\\D-85747 Garching b.\ M\"unchen\\ Germany}
\email{giannoulis@ma.tum.de}
\address[A.~Mielke]{Weierstra{\ss}-Institut f\"ur Angewandte Analysis und 
Stochastik\\ Mohrenstra{\ss}e 39\\ 10117 Berlin  
\and Institut f\"ur Mathematik, Humboldt-Uni\-ver\-si\-t\"at zu Berlin\\ 
Rudower Chaussee 25\\12489 Berlin\\ Germany}
\email{mielke@wias-berlin.de}
\address[C. Sparber]{Department of Applied Mathematics and Theoretical
Physics\\ 
CMS, Wilberforce Road\\ Cambridge CB3 0WA\\ England}
\email{c.sparber@damtp.cam.ac.uk}

\begin{abstract} We investigate the asymptotic behavior of solutions
to semi-clas\-si\-cal Schr\"odinger equations with nonlinearities of
Hartree type.  For a weakly nonlinear scaling, we show the validity
of an asymptotic superposition principle for slowly modulated highly
oscillatory pulses.  The result is based on a high-frequency
averaging effect due to the nonlocal nature of the Hartree
potential, which inhibits the creation of new resonant waves.  In
the proof we make use of the framework of Wiener algebras.
\end{abstract}
\subjclass[2000]{35B40, 35C20, 81Q20}

\keywords{Nonlinear Schr\"odinger equation, Hartree nonlinearity, high-frequency asymptotics, WKB approximation}

\thanks{C.S. has been supported by the Royal
Society via his University Research Fellowship. A.M. was partially
supported by DFG via project D14 of {\scshape Matheon}.} 
\maketitle


\section{Introduction and main result}

In this work we are interested in the asymptotic behavior for $0< \e \ll 1$ of
the following nonlinear Schr\"odinger equation
\begin{equation}
\label{nls}
i\e \partial _t {\tt u}^\e =   -\frac{\e^2}{2}\Delta {\tt u}^\e + \left ( K \ast |{\tt u}^\e|^2 \right) {\tt u}^\e,\quad {\tt u}^\e \big |_{t=0}={\tt u}^\e_0,\\
\end{equation}
where $(t,x)\in \R_+\times\R^d$, $d\in \N$. This model describes the
time-evolution of a complex-valued field ${\tt u}^\e(t,x)$, under the
influence of a \emph{Hartree-type nonlinearity} (cf. \cite{SuSu} for a broader introduction) 
Above, $K=K(x)$ denotes a given real-valued interaction kernel, to be specified in
detail below, and $` `\ast"$ denotes the convolution w.r.t. $x\in
\R^d$. The scaling of \eqref{nls} for small $\e >0$ corresponds to the
\emph{semi-classical} regime, i.e. the regime of \emph{high-frequency}
solutions ${\tt u}^\e(t,x)$, which can be approximated via geometric
optics.

The asymptotic behavior of \eqref{nls} as $\e \to 0_+$ has been
studied by several authors, mainly for the case $K (x) = \pm
\frac{1}{|x|}$: For example in \cite{LiPa, MaMa, Z}, Wigner measure
techniques are invoked, which however require mixed states and thus
can not be applied to our situation. In the one-dimensional case, this
constraint can be overcome \cite{ZZM}, but uniqueness of the limiting
Wigner measure for $t>0$ is still open. Turning to multi-scale WKB
expansions, which are typically valid for short times only, a variety
of asymptotic results can be found in \cite{AlCa, CaMa, HaLi, LiTa},
provided that ${\tt u}_0^\e$ is given as a \emph{single-phase} WKB
initial data, i.e
$$
{\tt u}^\e _0(x)= a^\e(x)e^{i \varphi(x)/\e },
$$ 
with given ($\e$-independent) phase $\varphi(x)\in \C$ and amplitude
$a^\e(x)\in \R$,  such that, asymptotically, $a^\e \sim a_0 + \e a_1+
\e^2a_2+\dots$.  

In the present work we are interested in generalizing these studies to the case of 
(a superposition of) \emph{several WKB waves}. Due to the expected nonlinear interaction between high-frequency waves (i.e. the appearance of resonances), 
this problem is notoriously difficult, even on a formal level. 
We shall therefore simplify the situation considerably by turning our attention to the 
case of (asymptotically) \emph{small initial data} corresponding to \emph{modulated plane-waves}. More precisely, we consider ${\tt u}_0^\e(x)= \e^{\alpha/2} u_0^\e(x)$, where $\alpha \ge 1$ and 
$u_0^\e(x)$ is given by 
a superposition of $\e$-oscillatory plane-waves, i.e. 
\begin{equation}\label{data}
u^\e_0(x) = \sum_{j=1}^J a_j(x) e^{i k_j \cdot x / \e }  ,\quad k_j \in \R^d,
\end{equation}
with amplitudes $a_j(x)\in \C$, $j\in  \{1, \dots, J \} \subseteq \N$. Here, and in the
following, we shall assume for the sake of simplicity that the initial
amplitudes $a_j$ do not depend on $\e$, since we shall only be
interested in the leading order asymptotics. A generalization to
amplitudes admitting an asymptotic expansion in $\e$ is then
straightforward.  Without restriction of generality we shall also
assume that
$$
k_j\not = k_\ell \ \text{for all} \  
j\not = \ell \in \Gamma:= \{1, \dots, J \} \subseteq \N, 
$$
where $J\in \N \cup \{\infty\}$ and $\{1,\dots,\infty\}$ means
simply $\N$. 

By rescaling $u^\e= \e^{-\alpha/2} {\tt u}^\e$, we can rewrite the
considered model into
\begin{equation}
\label{hartree}
i\e \partial _t {u}^\e =   -\frac{\e^2}{2}\Delta {u}^\e + \e^\alpha \left ( K
\ast |{ u}^\e|^2 \right) {u}^\e,\quad { u}^\e \big |_{t=0}={u}^\e_0,
\end{equation}
where the initial data is now of order one but the equation displays
an (asymptotically) small nonlinearity. We shall from now on take on
this point of view since it looks more natural from the physical point
of view. In addition, it is well known (see e.g. \cite{Ca1, Ca2}) that
the choice $\alpha = 1$ for \eqref{hartree} is \emph{critical} as far
as semi-classical behavior is concerned (see Section \ref{sec:approx}
for more details). We shall therefore pay particular attention to this
case.  We remark that the same asymptotic scaling has been used in
\cite{CDSp, CMS} for Schr\"odinger equations with (gauge invariant)
power-law nonlinearities $\propto |u|^{2\sigma}u$, $\sigma \in \N$. In
particular, in \cite{CDSp} the problem of high-frequency wave
interaction is exhaustively studied in the case $\sigma = 1$ for which
a geometric description of all possible nonlinear resonances is given.

However, for the considered case of Hartree type nonlinearities, the
situation is very different, due to the \emph{nonlocal nature} of
$K\ast |u^\e|^2$.  Since we expect the solution $u^\e(t,x)$ to be
given asymptotically as a superposition of highly oscillatory waves,
we clearly can not regard the Hartree term to be slowly varying (as in
the case of a single wave). The notion of a resonance though, seems to
be not clearly defined in this case. A more sophisticated analysis of
the oscillatory structure of $K\ast |u^\e|^2 $ is needed therefore. As
we shall see, the nonlocal nature of the Hartree potential yields a
kind of \emph{averaging effect}. In particular \emph{no new} highly
oscillatory phases are created in leading order (via resonances), in
sharp contrast to the situation for local nonlinearities.

In order to be more precise, we need to introduce some notation: 
The Fourier transform of $f\in L^1\cap L^2(\R^d)$ will be denoted as
\begin{equation*}
(\mathcal F f)(\xi)\equiv \widehat f(\xi)=\frac{1}{(2\pi)^{d/2}}\int_{\R^d}f(x)e^{-ix\cdot \xi} dx.
\end{equation*}
Our analytical approach will be based on the use of the \emph{Wiener Algebra} (see Section \ref{sec:wiener} for more details). Within this framework it turns out 
that the natural space for the amplitudes $a=(a_j)_{j\in \Gamma}$ is given as follows.
\begin{definition}\label{Aspace}
\begin{equation*}
\mathcal A (\R^d) := \{ a=(a_j)_{j\in\Gamma} : (\widehat{a}_j)_{j\in\Gamma}\in
\ell^1(\Gamma;L^1(\R^d)) \},
\end{equation*}
$\Gamma= \{1, \dots, J \} \subseteq \N$,  
equipped with the norm
\begin{equation*}
\|a\|_{\mathcal A (\R^d)} = \sum_{j=1}^J\|\widehat{a}_j\|_{L^1(\R^d)}.
\end{equation*}
\end{definition}
We are now in the position to state the main theorem of this work.
\begin{theorem}\label{theorem} 
For $d\ge 1$, consider the Cauchy problem \eqref{hartree} with
$\alpha \ge 1$, subject to initial data $u^\e_0$ of the form
\eqref{data}, where the initial amplitudes $a_j\in L^2(\R^d)\cap
\calA(\R^d)$ are such that $(\partial^p_x
a_j)_{j\in\Gamma}\in\calA(\R^d)$ for all $|p|\le 2$. In addition assume 
\begin{equation}\label{small} 
|\Lambda|_\infty^{-1}:=\inf\{|k_\ell-k_m|:\ \ell,m\in\Gamma,\ \ell\neq m\}>0,
\end{equation}
and let the interaction kernel $K$ satisfy $(1 + |\xi|) \widehat 
K (\xi) \in L^\infty(\R^d)$.

Then, for all $T>0$ there exists $C,\e_0>0$, such 
that for any $\e\in(0,\e_0)$, the exact
solution to \eqref{hartree} can be approximated by
\begin{equation}\label{errorestimate}
\left\lVert \u-\uapp\right\rVert_{L^\infty([0,T]\times \R^d)}\le
C\e^\beta
\quad\text{with}\ \beta  =
\begin{cases} 1 & \text{if $\alpha\in\{1\}\cup[2,\infty)$,}\\
\alpha{-}1 & \text{if $\alpha\in(1,2)$.}
\end{cases} 
\end{equation} 
Here the approximate solution $u_{\rm app}^\eps$ is given by
\begin{equation}\label{approx}
u_{\rm app}^\e(t,x) = \sum_{j=1}^J a_j(x-t k_j)e^{i S_j(t,x)} e^{i k_j \cdot x/\e  - i t |k_j|^2 / (2\e) }  ,
\end{equation}
where $S_j(t,x)\in \R$ is defined in \eqref{S} if $\alpha=1$ and 
$S_j(t,x)=0$ if $\alpha >1$, respectively.
\end{theorem}

In \eqref{approx}, the total number of (highly
oscillatory) phases $J \in \N\cup{\infty}$, is the same as for the
initial data \eqref{data}. Hence, no new phases are created by the
nonlinearity.  Nonlinear effects only show up  in leading order via
\emph{self-modulation of the amplitudes} (provided $\alpha = 1$).
The condition \eqref{small} can be seen as a small
divisor assumption required in the case of infinitely many phases, though, \emph{not of the same type} as the one used in \cite{CDSp}. Obviously, 
\eqref{small} is always satisfied if $J < + \infty$. 

Note that if $\alpha \in \N$, the error of our approximation is at least of the order $\O(\e)$ and thus smaller than the size of the original 
(small amplitude) solution. However, we see from \eqref{errorestimate} that if $\alpha \in (1,2)$ the error becomes worse as $\alpha \to 1_+$. 
This problem can in principle be overcome if one allows the leading order amplitudes to be mildly dependent on $\e$ themselves, cf. Remark \ref{SchubertRemark} below.

Finally, one should note that under the general assumptions of this work, we can not infer global well posedness of equation \eqref{hartree} in the Wiener
space. The usual arguments for proving global existence (see e.g. \cite{SuSu}) involve the
conservation of the $L^2$ norm, which also holds in our case. However,
this is not sufficient to control the nonlinearity $(K *|u|^2)u$ in the
Wiener space. Nevertheless, the above theorem shows
that for initial data $u^\e_0$ in the form \eqref{data}, the solutions
cannot blow up too fast: If $T^\e>0$ denotes the blow-up time, then $T^\e\to
+\infty$ as $\e \to 0_+$.

\begin{remark} A particular example for $K$, satisfying the
 assumptions is given by the one considered in \cite[equation
 (15)]{Be} and more importantly by the so-called \emph{Yukawa
   potential}
$$ K(x) = \pm \frac{e^{-\lambda |x|}}{|x|}, \quad   x \in \R^3, \lambda >0,$$
for which the corresponding Fourier transform is found to be
$$
\widehat K(\xi) = \pm \frac{1}{\lambda^2+|\xi|^2}.
$$
Clearly, $(1 + |\xi|) \widehat K \in L^\infty(\R^3)$ in this case. Note that in the limit $\lambda \to 0$ we obtain the Fourier transform of the Coulomb potential $K(x) = \pm  \frac {1}{|x|}$, which, however, is too singular, 
to directly apply our theorem. It remains an interesting open problem to establish the same result for the Coulomb case in $d=3$.
\end{remark}

The paper is 
organized as follows: In Section \ref{sec:approx}, we formally derive the approximate solution and draw some further conclusions from it. In Section \ref{sec:wiener} we set up 
the Wiener framework for the exact and the approximate solution. Finally, we prove the required estimates on the remainder of the approximation and consequently 
state the proof of Theorem \ref{theorem} in Section \ref{sec:proof}.

\section{Derivation of the approximate solution}\label{sec:approx}

We seek an approximation of solutions to \eqref{hartree} in the form
\begin{equation}\label{ansatz}
u_{\rm app}^\e(t,x) = \sum_{j=1}^J A_j(t,x) e^{i \phi_j(t,x)/ \eps }  .
\end{equation}
Assuming for the moment sufficient smoothness for $A_j$ and $\phi_j$, we can plug this ansatz into \eqref{hartree}, which yields 
\begin{equation}\label{eq:approx}
i \e \partial _t u_{\rm app}^\e  +\frac{\e^2}{2}\Delta  u^\e_{\rm app}  -
\e^\alpha \left ( K \ast |u^\e_{\rm app}|^2 \right) u^\e_{\rm app} = 
\sum_{n=0}^2 \e^n X_n +\e^\alpha(Y+Y_R),
\end{equation}
where we denote
\begin{align}
\notag
X_0 &= -\sum_{j=1}^J e^{i \phi_j/ \e}
\Big(\partial_t \phi_j + \frac{1}{2} |\nabla \phi_j |^2 \Big) A_j ,
\\
\notag
X_1 &= i\sum_{j=1}^J e^{i \phi_j/ \e} 
\Big(\partial_t A_j+\nabla A_j\cdot\nabla\phi_j+\frac12A_j\Delta\phi_j\Big),
\end{align}
and also
\begin{align}
\label{res0}
X_2 &= \frac12\sum_{j=1}^J e^{i \phi_j/ \e} \Delta A_j.
\end{align}
These terms are the same as in the linear case $K\equiv0$. 
Due to the presence of the Hartree type nonlinearity, we also obtain
\begin{align}
\label{resY}
Y   &= -\sum_{j=1}^J e^{i \phi_j/ \e}
\Big(K\ast\sum_{\ell=1}^{J}|A_{\ell}|^2\Big) A_{j} ,\\
\label{res}
Y_R &= -\sum_{j=1}^J e^{i \phi_j/ \e}   
\Big(K\ast\sum_{\ell,m=1 \atop \ell \not = m }^{J} 
\big(A_{\ell} \overline{A}_{m} \, e^{i (\phi_\ell - \phi_m)/ \e}\big)
\Big) A_{j}.
\end{align}
Obviously, $Y_R$ carries high-frequency oscillations, which are not
captured by our ansatz \eqref{ansatz}. Thus we need to develop a more
careful analysis in the following, which shows that $Y_R$ is of
higher order. 

Ignoring this problem for the moment, we consequently aim to eliminate
equal powers of $\e$. Hence, in leading order, we set $X_0 = 0$, which
is equivalent to the \emph{Hamilton-Jacobi equation}
\begin{equation}\label{hj}
\partial_t \phi_j + \frac{1}{2} |\nabla \phi_j |^2 =0, \quad \phi_j
\big |_{t=0}= k_j\cdot x. 
\end{equation}
Solutions to \eqref{hj} determine the characteristic high-frequency
oscillations present in $u_{\rm app}^\e$. In our case, they are easily
found to be
\begin{equation}\label{eq:planewave}
\phi_j(t,x) =k_j\cdot x -\frac{t}{2}|k_j|^2.
\end{equation}
These phases obviously solve \eqref{hj} for all $(t, x)\in \R\times \R^d$, i.e. \emph{no caustics} appear in our study. 

In the next step we set $X_1=0$ if $\alpha>1$ and $X_1+Y=0$ if
$\alpha=1$ (note that we do not include $Y_R$ here).  Comparing
the prefactors of the terms multiplied by $e^{i\phi_j/\e}$, yields
the following system of transport equations for the amplitudes:
\begin{equation}\label{eq:ampsyst}
\partial_t A_{j} + k_j \cdot \nabla  A_{j}    =  
\begin{cases}
0 & \text{ if $ \alpha > 1$,}\\
-i V_{\rm eff}(A)A_{j}  & \text{ if $ \alpha = 1$,} 
\end{cases}
\end{equation}
where we have used the fact that $\Delta \phi_j \equiv 0$, in view of
\eqref{eq:planewave}. For $\alpha = 1$, the effective (nonlinear)
potential $V_{\rm eff} (A)$ is given by
$$
V_{\rm eff} (A):=K \ast  \Big(\sum_{\ell=1}^J |A_{\ell}|^2\Big ).
$$
We see that for $\alpha >1$ no nonlinear effects are present in
transport equations for the leading order amplitudes. The case $\alpha
= 1$ is therefore seen to be critical as far as semi-classical
asymptotics is concerned.
\begin{lemma} \label{lem:repAj} The transport equation
\eqref{eq:ampsyst} with initial data $(a_j)_{j\in\Gamma}\in L^2\cap
\calA(\R^d)$ admits global-in-time solutions $A \in C([0,\infty);
L^2\cap \calA(\R^d))$, which can be written in the form
\begin{equation}\label{amp}
A_j(t,x)= a_j(x-t k_j) e^{i S_j(t,x)},
\end{equation}
where $S_j\equiv 0$ for $\alpha>1$  and $S_j\in C([0,\infty)\times
\R^d)$ for $\alpha=1$ is given by  
\begin{equation}
\label{S}
S_j(t,\cdot)
=-\int_0^t  \Big(K\ast\sum_{\ell=1}^J 
\big|a_\ell\big(\cdot+(\tau-t)k_j-\tau k_\ell\big)\big|^2  \Big)
\,d\tau .
\end{equation}
In particular we have mass conservation for each individual mode 
\begin{equation}\label{con}
\|A_j(t,\cdot)\|_{L^2(\R^d)} = \|a_j \|_{L^2(\R^d)}, \quad \forall \, t \in \R.
\end{equation}
\end{lemma}
In contrast to $\phi_j$ the phases $S_j$ are only
\emph{slowly varying}, i.e. they oscillate with frequencies larger than $1/
\e$. They describe the nonlinear self-modulation of the amplitudes
but do not show up in quadratic quantities, like the mass density
$|A_j|^2$ etc.
\begin{proof} 
By multiplying \eqref{eq:ampsyst} with $\bar A_j$ and taking the 
real part, we see 
\begin{equation*}
\(\d_t+ k_j \cdot \nabla \) |A_j|^2 =0,
\end{equation*}
which yields $ |A_j(t,x)|^2 =|a_j(x-t k_j)|^2 $ and thus \eqref{amp}
and \eqref{con}.  Finally, inserting \eqref{amp} into
\eqref{eq:ampsyst} and integrating along characteristics, we obtain
the expression \eqref{S} for $S_j$.
\end{proof}
Having obtained the characteristic phases
$\phi_j$ and the leading order amplitudes $A_j$ we shall now turn our attention towards the remainder, i.e.
\begin{equation}\label{remainder}
R(u_{\rm app}^\e) =
\begin{cases} \e^2 X_2+\e Y_R& \text{if $\alpha=1$,}\\
\e^2 X_2+ \e^\alpha (Y+Y_R)& \text{if $\alpha>1$.}
\end{cases} 
\end{equation}
For $\alpha=1$ the term $Y_R$ appearing within 
$R(u_{\rm app}^\e)$ is
formally of order $\O(\eps)$. Thus, at first glance, $Y_R$ seems to be too large to be considered a part of the remainder. 
It will be our main task to show that $Y_R$ is indeed sufficiently small as $\eps \to 0_+$. To this end, we shall rely on the framework of Wiener algebras.

\begin{remark}\label{SchubertRemark} 
In the case $\alpha >1$, there is a some freedom in deriving the system of amplitude equations. 
Indeed, one could regard $\eps^{\alpha-1}$ 
merely as a prefactor for the interaction kernel $K$ and solve in the next step $X_1+\eps^{\alpha-1}Y=0$. Instead of \eqref{eq:ampsyst}, 
we would consequently get
\begin{equation*}
\partial_t A^\e_{j} + k_j \cdot \nabla  A^\e_{j}    =  
-i \eps^{\alpha-1}V_{\rm eff}(A^\e)A^\e_{j} \quad \text{ for  $\alpha\ge1$,} 
\end{equation*}
which yields
\begin{equation}\label{epsAmp}
A^\e_j(t,x) =  a_j(x-t k_j) e^{i \e^{\alpha -1} S_j(t,x)},
\end{equation}
with $S_j$ given as before. In other words, the leading order amplitudes would become $\e$-dependent and instead of \eqref{remainder} the remainder would have the form   
\begin{equation*}
R(u_{\rm app}^\e) =\e^2 X_2+ \e^\alpha Y_R\quad \text{for $\alpha\ge1$.}
\end{equation*}
We expect that the subsequent analysis can be adapted to the case of 
$\eps$-dependent amplitudes given by \eqref{epsAmp}, and in fact we expect a slightly better error estimate in this case. Namely, we expect \eqref{errorestimate} to hold
with $\beta=1$ \emph{for all} $\alpha\ge1$. However, we hesitate to carry through this approach in full detail, since 
the separation of slow- and fast scales within the leading order approximation becomes much less apparent. In addition, one usually 
does not regard \eqref{epsAmp} as a proper leading order amplitude in semi-classical analysis (due to its $\e$-dependence). 
We shall therefore resume our analysis with \emph{$\eps$-independent $A_j$} as in our original ansatz \eqref{ansatz}.
\end{remark}

\section{The Wiener algebra framework}
\label{sec:wiener}

We now present the analytical framework of Wiener algebras which already
proved its use in similar circumstances, cf.\ \cite{CDSp, MColinLannes, JMRWiener}. We start with the following definition.
\begin{definition}[Wiener Algebra]
We define
\begin{equation*}
W(\R^d):= \left\{ f\in {\mathcal S}'(\R^d;\C),\ \
  \|f\|_W:=\|\widehat f\|_{L^1(\R^d)} <\infty\right\}. 
\end{equation*}
\end{definition}
The following properties of $W(\R^d)$ have been proved in \cite{MColinLannes, JMRWiener}.

\renewcommand{\labelenumi}{\textbf{\roman{enumi}.}}

\begin{lemma}\label{lem:propWiener}\ 
\begin{enumerate}
\item \label{Assert1}
$W(\R^d)$ is a Banach space, continuously embedded into $L^\infty(\R^d)$. 
\item\label{Assert2}
$W(\R^d)$ is an algebra, in the sense that the mapping $(f,g)\mapsto fg$ 
is continuous from $W(\R^d)^2$ to $W(\R^d)$, and moreover
\begin{equation*}
 \forall f,g\in W(\R^d), \quad
 \|fg\|_W\le \|f\|_W\|g\|_W.
\end{equation*}
\item \label{Assert4}
Let $U^\eps(t)=e^{i \eps \frac{t}{2}\Delta}$ denote the free Schr\"odinger 
group. Then, for all $t\in \R$, $U^\eps(t)$ is unitary on $W(\R^d)$. 
\end{enumerate}
\end{lemma}
Assertion iii. follows from the fact that $\widehat U^\eps(t) = e^{i t |\xi|^2/2}$, acting as a multiplication operator in  Fourier space.
From now on, we shall consider the Cauchy problem \eqref{hartree} to be posed in $W(\R^d)$. To this end we need the following well-posedness result 
(which is an adaptation of the one given in \cite{CDSp} to the case of Hartree nonlinearities).

\begin{lemma}\label{lem:existsol}
Consider the initial value problem 
\begin{equation} \tag{\ref{hartree}}
\label{ivp}
i\eps\d_t u^\eps +\frac{\eps^2}{2}\Delta u^\eps =  \eps^\alpha
(K\ast |u^\eps|^{2})u^\eps, \quad u^\eps\big |_{t=0}=u_0^\eps,
\end{equation}
with $\alpha \ge 1$.  If $\widehat K \in L^\infty(\R^d)$ and $u_0^\eps
\in W( \R^d)$ with $\|u_0^\eps\|_W\le D_0$, then there exists a
$T_0>0$, which depends on $D_0$ but not on $\e$, and a unique solution
$u^\eps \in C([0,T_0];W( \R^d))$ of \eqref{ivp}.
\end{lemma}
\begin{proof}
Duhamel's formulation of \eqref{ivp} reads
\begin{equation*}
 u^\eps(t) = U^\eps(t) u_0^\eps - i \eps^{\alpha-1} \int_0^t
 U^\eps(t-\tau)\((K\ast |u^\eps|^{2})u^\eps)\) (\tau) \, d\tau. 
\end{equation*}
Denote, for fixed $\u_0$, the right hand side of this formula 
by $\Phi^\eps(u^\eps)(t)$. From  Lemma~\ref{lem:propWiener} iii. we have
\begin{equation}\label{phi1}
\left\lVert \Phi^\eps(u^\eps)(t)\right\rVert_{W}\le  D_0 
+\eps^{\alpha-1}\int_0^t \|\((K\ast |u^\eps|^{2})u^\eps)\)(\tau) \|_W \, d\tau.
\end{equation}
In order to control the nonlinear term, we need to estimate expressions of the form $(K\ast(uv))w$ in $W(\R^d)$. 
To this end, we first use Lemma \ref{lem:propWiener} ii, to obtain
\begin{equation*}
\|(K\ast(uv))w\|_W\le \|K\ast(uv)\|_W\|w\|_W.
\end{equation*}
By H\"older's inequality we also get
\begin{equation}\label{mainestimate2}
\|K\ast(uv)\|_W=\|\widehat{K\ast(uv)}\|_{L^1} \le \|\widehat K \|_{L^\infty}
\|uv\|_W,
\end{equation}
and applying again Lemma \ref{lem:propWiener} ii, we arrive at
\begin{equation}\label{mainestimate}
\|(K\ast(uv))w\|_W
\le \|\widehat K\|_{L^\infty}\|u\|_W\|v\|_W\|w\|_W
\quad\text{for $u,v,w\in W(\R^d)$}.
\end{equation}
Thus, from \eqref{phi1} we obtain
\begin{equation*}
\left\lVert \Phi^\eps(u^\eps)(t)\right\rVert_{W}\le D_0 
+\eps^{\alpha-1}\|\widehat K\|_{L^\infty}\int_0^t\|\u(\tau)\|_W^3\, d\tau.
\end{equation*}
Moreover, $\u \mapsto \Phi^\eps(u^\eps)$ is locally Lipschitz in
$U:=C([0,T],W(\R^d))$: If $\|u^\eps\|_{U}\le D$, 
$\|v^\eps\|_{U}\le D$, then there exists $C=C(D)$ such that
\begin{equation*}
\left\lVert \Phi^\eps(u^\eps)(t)- \Phi^\eps(
 v^\eps)(t)\right\rVert_{W}\le C(D)\int_0^t\left\lVert
 u^\eps(\tau)- v^\eps(\tau)\right\rVert_{W}
d \tau,\quad \forall t\in [0,T]. 
\end{equation*}
A fixed point argument in $ \big\{ u\in U:\ \sup_{t\in
 [0,T]}\|u(t)\|_W\le D \big\}$, with $D>D_0$, for $T=T_0$
sufficiently small, then yields Lemma~\ref{lem:existsol}.
\end{proof}

Having set up an existence result for the exact solution $u^\e$ in
$W(\R^d)$, we now turn to the approximate solution $\uapp$ given by
\eqref{approx}. To this end, we shall need the following lemma, 
which shows that the Wiener space is perfectly adapted to our
problem. 

\begin{lemma} \label{lem:substitR}
Let $k_j \in\R^d$, $c_j\in\R$, and $b \in \ell^1(\Gamma, W(\R^d))$, and
set 
\begin{equation*}
 f(x,y)=\sum_{j=1}^J b_j(x) e^{i (k_j \cdot y+c_j)}.
\end{equation*}
Then, for all $\eps>0$ the function $f(\cdot,\cdot/\eps): x
\mapsto f(x,x/\e)$ lies in $W(\R^d)$ 
with
$$
\|f(\cdot,\cdot/\eps)\|_W \le \| b \|_{\mathcal A} 
= \sum_{j\in \Gamma} \| b_j\|_W  .
$$
\end{lemma}
\begin{proof}
We write 
$$\|f(\cdot,\cdot/\eps)\|_W = \big \lVert 
\sum_{j=1}^Je^{ic_j}\widehat{b_j}(\cdot-k_j/\eps)\big \rVert_{L^1}  
\le \sum_{j=1}^J \|\widehat{b_j}(\cdot-k_j/\eps)\|_{L^1} 
= \sum_{j=1}^J \|\widehat{b_j}\|_{L^1}.$$
The last term is, by definition, $\|b \|_{\mathcal A}$. 
\end{proof}

\begin{remark} This lemma in general does
\emph{not} hold for functions of 
the form $f(x,y)= \sum_{j=1}^J b_j(x)  e^{i \varphi_j(y)}$, with 
$\varphi_j(y) \not = k_j \cdot y+c_j$. 
A generalization of our study to non-plane wave WKB phases therefore seems to be a delicate issue (at least within the Wiener framework) and by no means straightforward. 
\end{remark}

Since the phases $\phi_j$ considered in this work are of plane-wave form 
\eqref{eq:planewave}, 
applying Lemma \ref{lem:substitR} with $c_j=-\frac{t}{2\e}|k_j|^2$
to \eqref{approx},
we immediately obtain
\begin{equation}\label{uappagain}
\|\uapp(t,\cdot)\|_W 
\le \|A(t,\cdot)\|_\calA.
\end{equation}
Similarly, we can estimate the expression \eqref{res0} for $X_2$ by
\begin{equation}\label{estuappX2}
\|X_2(t,\cdot)\|_W  
\le \frac12 \|\Delta A(t,\cdot)\|_\calA = \frac12 \sum_{j=1}^J\|\Delta A_j(t,\cdot)\|_W.
\end{equation}
In addition, we obtain the following estimate for $Y$, defined in \eqref{resY}:
\begin{equation*}
\|Y(t,\cdot)\|_W
\le
\Big\|K\ast\sum_{\ell=1}^{J}|A_{\ell}(t,\cdot)|^2\Big\|_W
\|A(t,\cdot)\|_\calA,
\end{equation*}
where we have used Lemma \ref{lem:substitR} and  
Lemma \ref{lem:propWiener} ii. This can be estimated further similarly to \eqref{mainestimate2} by using 
Lemma \ref{lem:substitR} (with $k_j=0$, $c_j=0$), as well as Assertion ii. of Lemma \ref{lem:propWiener} and the fact that $\ell^\infty\subset\ell^1$, to obtain 
\begin{equation}\label{estY}
\|Y(t,\cdot)\|_W
\le
\|\widehat K\|_{L^\infty}
\|A(t,\cdot)\|_\calA^3.
\end{equation}
In order to close the argument, we consequently require appropriate bounds in $\mathcal A(\R^d)$ on the amplitudes $A_j(t,\cdot)$, together with their spatial derivatives. 

\begin{lemma} \label{lem:amp}
Let $\alpha \ge 1$ and $\widehat K \in L^\infty (\R^d)$. 
For all $a=(a_j)_{j\in\Gamma}\in \mathcal A(\R^d)$, 
there exists a unique solution $A \in C([0,\infty); \calA(\R^d))$ 
to the system~\eqref{eq:ampsyst}. Moreover, if $(\partial^p_x a_j)_{j\in\Gamma}\in\calA(\R^d)$, for 
$|p|\le 2$, 
then $(\partial^p_x A_j)_{j\in\Gamma}\in C([0,\infty);\calA(\R^d))$.
\end{lemma} 

\begin{proof} 
For $\alpha>1$ the statements of the lemma are immediately clear, since in
this case  
$A_j(t,x)=a_j(x-tk_j)$, cf.\ \eqref{eq:ampsyst}. 
For $\alpha=1$ we rewrite \eqref{eq:ampsyst} in its integral form
\begin{equation} \label{eq:transportsystemint}
A_j(t,x) = a_j(x-t  k_j)  
+\int_0^t \mathcal{N}(A)_j(\tau,x+(\tau-t)k_j) d\tau,
\end{equation}
where the nonlinearity $\mathcal{N}(A)_j$ is given by
\begin{equation*}
\calN(A)_j  = 
\ds-i\Big(K\ast\big(\sum_{\ell=1}^J|A_{\ell}|^2\big)\Big )A_j
\end{equation*}
Invoking the same arguments as for the derivation of \eqref{estY}, we obtain
\begin{equation*}
\|\calN(A)\|_\calA=\sum_{j=1}^J\|\calN(A)_j\|_W
\le \|\widehat K\|_{L^\infty}\|A\|_\calA^3.
\end{equation*}
This shows that $\mathcal{N}(A) $ defines a continuous mapping from $\mathcal A^{3} $ to $\mathcal A$ and a local-in-time existence result immediately
follows from the standard Cauchy-Lipschitz theorem for ordinary differential equations. That the solutions $A_j$ indeed exist for all $t\ge 0$ then follows from the 
explicit representation \eqref{amp}.
From the latter we additionally obtain the propagation of regularity,
by explicit calculation of $\partial_x^p A$.
\end{proof}

Lemma \ref{lem:amp} consequently establishes the estimates in
$W(\R^d)$ for $u_{\rm app}^\e$, $X_2$ and $Y$ in a rigorous way.  Note
however, that the above given estimates do not yield an estimate for
the remainder $R(u^\e_\text{app})$, given by \eqref{remainder},
since it also includes $Y_R$, which we completely ignored so
far.  We will make up for it in the following section.

\section{Estimates on the remainder  
and proof of the main theorem} 
\label{sec:proof} 

It remains to estimate in $W(\R^d)$ the term $Y_R$ given in \eqref{res}. To this end we shall prove the following key technical result.

\begin{proposition} \label{lem:remainder}
Let $Y_R$ be defined by \eqref{res} 
with plane-wave phases $\phi_j$ given by \eqref{eq:planewave} 
and assume 
$|\Lambda|_\infty^{-1}:=\inf\{|k_\ell-k_m|:\ \ell,m\in\Gamma,\ \ell\neq m\}>0$. Moreover, let $K$ be such that $\widehat K \in L^\infty (\R^d)$ and $\widehat { \nabla K} \in L^\infty(\R^d)$. 
Then we have the following bound:
\begin{equation}\label{Ybound}
\| Y_R (t,\cdot)\|_{W}  \le \e\, C_K \|  A(t,\cdot) \|^2_\calA 
\big(\|  A (t,\cdot)\|_\calA +\| \nabla A (t,\cdot)\|_\calA \big),
\end{equation}
where $C_K>0$ is independent of $\eps$.
\end{proposition} 
\begin{proof} 
Recalling the definition of $Y_R$ given in \eqref{res} and taking into account
the particular plane-wave form \eqref{eq:planewave} of the phases $\phi_j$,
we obtain from Lemma \ref{lem:substitR} and Assertion ii. of Lemma 
\ref{lem:propWiener} that
\begin{equation}\label{YRfirstest}
\|Y_R\|_W \le 
\|A\|_\calA
\Big\|\sum_{\ell,m=1 \atop \ell \not = m }^{J} 
K\ast\big(A_{\ell}\overline{A}_{m}e^{i(\phi_\ell-\phi_m)/\e}\big)\Big\|_W.
\end{equation}
Using,
$$e^{i y \cdot k/ \e+z}=-i\e \frac{k}{|k|^2}\cdot\nabla_y e^{i y \cdot k/\e+z}
\quad (z\in\C),$$ 
we can perform a partial integration w.r.t $y$, and rewrite
\begin{multline*}
K\ast\big(A_\ell\bar A_m e^{i(\phi_\ell - \phi_m)/\eps}\big)
=\int_{\R^d} K(x-y)A_\ell(y)\bar A_m(y)
e^{i(\phi_\ell(y) - \phi_m(y))/\eps}\,dy 
\\
=i\e\int_{\R^d}\frac{k_\ell-k_m}{|k_\ell-k_m|^2}
\cdot\nabla_y\big(K(x-y)A_\ell(y)\bar A_m(y)\big)
e^{i(\phi_\ell(y)-\phi_m(y))/\e}\,dy.
\end{multline*}
To show that the boundary terms vanish, assume first that $A_j\in
\mathcal S(\R^d)$, the set of Schwartz functions (for which the
boundary terms clearly vanish). Since $\mathcal S(\R^d)$ is dense in
$L^1(\R^d)$, Fourier transformation implies that $\mathcal S(\R^d)$ is
also dense in $W(\R^d)$. Consequently, the fact that the expressions
on both sides are norm-bounded sesquilinear forms establishes the
above formula.

Using $k_\ell \not = k_m\in \R^d$, we know that   
$\Lambda_{\ell, m}:= \frac{k_\ell-k_m}{|k_\ell-k_m|^2} \in \R^d$
is well defined and consequently   
\begin{align*}
K\ast\big(A_\ell\bar A_m e^{i(\phi_\ell - \phi_m)/\eps}\big)
=
&-i\e(\Lambda_{\ell, m}{\cdot}\nabla K)\ast
\big(A_\ell\bar A_m e^{i(\phi_\ell-\phi_m)/\e}\big)
\\&
+i\e K\ast\big(\big(\Lambda_{\ell, m}{\cdot}\nabla (A_\ell\bar A_m)\big) 
e^{i(\phi_\ell-\phi_m)/\e}\big).
\end{align*}
Using the estimate \eqref{mainestimate2} and Lemma \ref{lem:substitR} 
(with $J=1$)
we get
\begin{align*}
\|K\ast\big(A_\ell\bar A_m e^{i(\phi_\ell - \phi_m)/\eps}\big)\|_W \le 
& \ \e
\|\Lambda_{\ell, m}{\cdot}\widehat{\nabla K}\|_{L^\infty}
\|A_\ell\|_W\|A_m\|_W  \\
& \, + \e  \|\widehat K\|_{L^\infty}
\|\Lambda_{\ell, m}{\cdot}\nabla (A_\ell\bar A_m)\|_W.
\end{align*}
Invoking again Lemma \ref{lem:substitR} (with $k_j=0$, $c_j=0$)
and Lemma \ref{lem:propWiener} ii, we conclude
\begin{align*}
&\Big\|\sum_{\ell,m=1 \atop \ell \not = m }^{J} 
K\ast\big(A_\ell\bar A_m e^{i(\phi_\ell - \phi_m)/\eps}\big)\Big\|_W
\\\quad &
\le 
\e\sum_{\ell,m=1 \atop \ell \not = m }^{J} \Big(
d|\Lambda_{\ell, m}|\|\widehat{\nabla K}\|_{L^\infty}
\|A_\ell\|_W\|A_m\|_W
+2\|\widehat K\|_{L^\infty}
|\Lambda_{\ell, m}|\|\nabla A_\ell\|_\calA\|A_m\|_W\Big)
\\\quad &
\le 
\e|\Lambda|_\infty \|A\|_\calA 
\Big(d\|\widehat{\nabla K}\|_{L^\infty}\|A\|_\calA
+2\|\widehat K\|_{L^\infty}\|\nabla A\|_\calA\Big)
\end{align*}
where
$\|\widehat{\nabla K}\|_{L^\infty}=\max\limits_{n=1,\ldots,d}
\|\widehat{\partial_{x_n}K}\|_{L^\infty}$, and
$$\|\nabla A\|_\calA=\sum\limits_{\ell=1}^J\|\nabla A_\ell\|_\calA
=\sum\limits_{\ell=1}^J
\sum\limits_{n=1}^d\|\partial_{x_n}A_\ell\|_W.$$
This, together with \eqref{YRfirstest} yields the estimate \eqref{Ybound}.
\end{proof}
\begin{remark}
Proposition \ref{lem:remainder} shows that  $\|Y_R\|_W=\O(\eps)$ and thus can indeed be considered a 
part of the remainder. Note that in the proof it is essential to invoke a 
stationary-phase type argument first, before starting to take estimates. 
In fact we would not succeed to show that $\|Y_R\|_W=\O(\eps)$, 
if we would estimate $Y_R$ in its original form. \end{remark}

By combining the results of Lemma \ref{lem:amp}  and Proposition \ref{lem:remainder}, 
we obtain the following result.  
\begin{corollary}\label{cor:estuappR}
Under the assumptions of Lemma \ref{lem:amp}
and Proposition \ref{lem:remainder}, there exists, for every $T>0$, 
a constant $C_R(T)>0$, 
independent of $\e>0$, such that the remainder $R(u_{\rm app}^\e)$ given by \eqref{remainder}
satisfies
\begin{equation}\label{estuappR}
\|R(u_{\rm app}^\e)\|_W \le \e^\gamma C_R
\qquad\forall\ t\in[0,T],\ \forall\ \e>0,
\end{equation}
where
\begin{equation*}
\gamma=
\begin{cases}2& \text{if $\alpha\in\{1\}\cup[2,\infty)$,}
\\ \alpha &\text{if $\alpha\in(1,2)$.}
\end{cases}
\end{equation*}
\end{corollary} 
\begin{proof}
Lemma \ref{lem:amp} guarantees the existence of the norms 
$\|A(t,\cdot)\|_\calA$, $\|\nabla A(t,\cdot)\|_\calA$, 
$\|\Delta A(t,\cdot)\|_\calA<\infty$ for all $t\in[0,\infty)$,
which are independent of $\e>0$ and continuous in $t$. Hence, taking their maximum over $t\in[0,T]$, 
we obtain 
\eqref{estuappR}
from \eqref{estuappX2}, \eqref{estY}, \eqref{Ybound} 
and the definition \eqref{remainder} of $R(u_{\rm app}^\e)$.
\end{proof}

With the estimate of Corollary \ref{cor:estuappR} on 
$R(u_{\rm app}^\e)$ at hand, 
we can finally state the proof of our main theorem, which follows the
basic ideas established in \cite{KiScMi92} for justifying the
nonlinear Schr\"odinger equation as a modulation equation for
dispersive waves.
\begin{proof}[Proof of Theorem \ref{theorem}] 
We consider a fixed $T>0$ and introduce the following scaled error $\r$
between the original solution $\u$ to \eqref{ivp} 
subject to the initial data \eqref{data} 
and the approximation \eqref{approx}:
$$\e^\beta\r:=\u{-}\uapp,$$
with a parameter $\beta>0$ to be specified below.  Hence, $\r(0)=0$.
From \eqref{uappagain} and Lemma \ref{lem:amp} we know that there
exists a constant $C_A>0$, independent of $\e$, such that $\|
u^\e_{\rm app} (t,\cdot) \|_W \le C_A$, for all $t\in [0,T]$.  Since
$\r(0)=0$, it follows $\|\u_0\|_W\le C_A$.  Consequently, for any
$D>C_A$, Lemma \ref{lem:existsol} yields the existence of a unique
solution $\u\in C([0,T_0],W(\R^d))$ for some $T_0>0$ with
$\|\u(t)\|_W\le D$ for $t \in [0,T_0]$.

Moreover, from \eqref{ivp} and \eqref{eq:approx} it follows 
that $\r$ satisfies
\begin{equation*}
i\e\d_t\r +\frac{\e^2}{2}\Delta\r
=\e^{\alpha-\beta}\big(\calM(\uapp+\e^\beta\r)-\calM(\uapp)\big)
-\e^{-\beta}R(\uapp) 
\end{equation*}
with $\calM(u)=(K\ast|u|^{2})u$ for $t\le\tau:=\min\{T_0,T\}$, and by
Duhamel's formula and Lemma~\ref{lem:propWiener} iii.  we obtain
\begin{equation}
\begin{split}\label{voc}
\|\r(t)\|_W\le 
&\ \e^{\alpha-\beta-1}\int_0^t 
\|\calM(\uapp(\tau)+\e^\beta\r(\tau))-\calM(\uapp(\tau))\|_W \,d\tau 
\\
&+\e^{-\beta-1}\int_0^t\|R(\uapp(\tau))\|_W\,d\tau,
\end{split}
\end{equation}
for all $t\le \tau$.
Writing 
$$\calM(u+r)-\calM(u)=\big(K\ast(u\bar r+\bar u r+|r|^{2})\big)(u+r)
+(K\ast|u|^2)r , $$
the estimate 
\eqref{mainestimate} gives
\begin{equation}\label{estMgen}
\|\calM(u+r)-\calM(u)\|_W
\le\|\widehat K\|_{L^\infty}(3\|u\|_W^2+3\|u\|_W\|r\|_W+\|r\|_W^2)\,\|r\|_W.
\end{equation}
Hence, replacing $u=\uapp$ and $r=\e^\beta\r$, and recalling
$\beta>0$, we obtain for any $C>0$ and $\e_0\in(0,1]$, such that
$3\e_0^\beta C_AC+\e_0^{2\beta} C^2 =C_A^2$, that
\begin{equation}\label{estM}
\|\calM(\uapp+\e^\beta\r)-\calM(\uapp)\|_W
\le \e^\beta C_\calM\,\|\r\|_W
\quad\forall\ \e\le \e_0,\ t\le\tau
\end{equation}
where $C_\calM:=4\|\widehat K\|_{L^\infty}C_A^2 $, as long as
$\|\r\|_W\le C$.

Inserting the bounds \eqref{estM} and \eqref{estuappR} 
into \eqref{voc}, 
and recalling that $\e_0\le 1$, $\alpha\ge 1$, $\tau\le T$,
we consequently obtain for 
$\beta\in(0,\gamma{-}1]$
\begin{equation*}
\|\r(t)\|_W\le C_R T
+C_\calM \int_0^t\|\r(\tau)\|_W\,d\tau 
\quad\forall\ \e\le \e_0,\ t\le\tau.
\end{equation*}
By Gronwall's lemma this yields
\begin{equation*}
\|\r(t)\|_W\le C_R T e^{C_\calM t} 
\quad\forall\ \e\le \e_0,\ t\le\tau.
\end{equation*}
Hence, setting above $C:=C_R T e^{C_\calM T}$ and $D:=C+C_A$ this
estimate guarantees that the solution $\u$ exists on the whole time
interval $[0,T]$, cf.\ the proof of Lemma \ref{lem:existsol}.
Moreover, recalling Lemma \ref{lem:propWiener} i, we finally obtain
the error estimate \eqref{errorestimate} which finishes the proof.
\end{proof}

\medskip 

\textbf{Acknowledgement.} The authors want to thank Robert Schubert for pointing out to them the alternative approach described in Remark 2.2.

\bibliographystyle{amsplain}

\begin{thebibliography}{99}


\bibitem{AlCa} T. Alazard and R. Carles, \emph{Semi-classical limit of Schr\"odinger-Poisson equations in space dimension $n\geq3$}. 
J. Diff. Equ. {\bf 233} (2007), no. 1, 24--275.

\bibitem{Be} N. N. G. Berloff, \emph{Nonlocal nonlinear Schroedinger equations as models of
superfluidity}, J. Low Temp. Phys. {\bf 116} (1999) no. 5/6, 1--22.

\bibitem{Ca1} R. Carles, \emph{Semi-classical analysis for nonlinear Schr\"odinger equations}.
World Scientific, Co. Pte. Ltd., Hackensack, NJ 2008.

\bibitem{Ca2} R. Carles, \emph{WKB analysis for nonlinear Schr\"odinger equations with potential}.
Comm. Math. Phys. 269 (2007), no. 1, 195-221.

\bibitem{CDSp}  R. Carles, E. Dumas, and C. Sparber, \emph{Multiphase weakly nonlinear geometric optics for nonlinear Schr\"odinger equations}, preprint {\tt arXiv:0902.2468}. 

\bibitem{CMS}  R. Carles, N. Mauser, and H.P. Stimming, \emph{(Semi)classical limit of the Hartree equation with harmonic potential}. 
SIAM J. Appl. Math. 66 (2005), no. 1, 29-56.

\bibitem{CaMa} R. Carles and S. Masaki, \emph{Semi-Classical analysis for Hartree equation}. Asymptot. Anal. {\bf 58} (2008), no. 4, 211--227.

\bibitem{MColinLannes} M. Colin and D. Lannes, \emph{Short pulses approximations in dispersive media}, SIAM J. Math. 
Anal., to appear. Preprint {\tt arXiv:0712.3940}.

\bibitem{GiMiSp} J. Giannoulis, A. Mielke, and C. Sparber, \emph{Interaction of modulated pulses in the 
nonlinear Schr\"odinger equation with periodic potential}. J. Diff. Equ. {\bf 245} (2008), no. 4, 939--963. 

\bibitem{HaLi} H. Li and C.-K. Lin, \emph{Semi-Classical limit and well-posedness of nonlinear Schr\"odinger-Poisson systems}, Electronic J. Diff. Equ. {\bf 2003} (2003), no. 93, 1--17.

\bibitem{Ho} L. H\"ormander, \emph{The analysis of linear partial differential operators I}, Springer Verlag, Berlin 1983.

\bibitem{JMRWiener} J.-L. Joly, G. M\'etivier, and J. Rauch,
\emph{Coherent nonlinear waves and the Wiener algebra},
Ann. Inst. Fourier {\bf 44} (1994), no. 1, 167--196.

\bibitem{KiScMi92} P. Kirrmann, G. Schneider, and A. Mielke, \emph{The
  validity of modulation equations for extended systems with cubic
  nonlinearities}, Proc. Roy. Soc. Edinburgh Sect. A, {\bf 122}
(1992), 85--91,

\bibitem{LiPa} P.-L. Lions and T. Paul, \emph{Sur les mesures de Wigner}, Rev. Mat. Iberoamericana {\bf 9} (1993), no. 3, 553--618. 

\bibitem{LiTa} H. Liu and E. Tadmor, \emph{Semi-Classical limit of the nonlinear Schr\"odinger-Poisson equation with subcritical initial data}, 
Methods Appl. Anal. {\bf 9} (2002), no. 4, 470--484. 

\bibitem{MaMa} P. A. Markowich and N. J. Mauser, \emph{The classical limit of a self-consistent quantum-Vlasov 
equation in 3D}, Math. Models Methods Appl. Sci. {\bf 3} (1993), no. 1, 109--124. 

\bibitem {SuSu} C. Sulem and P.~L. Sulem, \emph{The nonlinear Schr\"odinger equation}. Springer Series on Applied Math. Sciences 139, Springer (1999).    

\bibitem {Z}  P. Zhang, \emph{Wigner measure and the semi-classical limit of Schr\"odinger-Poisson equation}, 
SIAM J. Math. Anal.  {\bf 34} (2002), no. 3, 700--718. 

\bibitem {ZZM}  P. Zhang, Y. Zheng, and N. J. Mauser, \emph{The limit from the Schr\"odinger-Poisson to the Vlasov- 
Poisson equations with general data in one dimension}, Comm. Pure Appl. Math. {\bf 55} (2002), 
no. 5, 582--632.

\end{thebibliography}

\end{document}